\documentclass[11pt, a4paper, twoside]{amsart}
\usepackage{mathrsfs}
\usepackage{amssymb}
\usepackage{amsfonts, graphicx, color}

\newtheorem{theorem}{Theorem}
\theoremstyle{plain}

\newtheorem{corollary}{Corollary}

\newtheorem{lemma}{Lemma}

\numberwithin{equation}{section}
\begin{document}
\title[Boundedness of the bilinear Bochner-Riesz Means]{ Boundedness of the
bilinear Bochner-Riesz Means\\
in the non-Banach triangle case}
\author{Heping Liu}
\address{School of Mathematical Sciences, Peking University, Beijing,
100871, P. R. China, }
\email{hpliu@math.pku.edu.cn}
\author{Min Wang}
\address{School of Mathematical Sciences, Peking University, Beijing,
100871, P. R. China, }
\email{wangmin09150102@163.com}
\date{}

\begin{abstract}
In this article, we investigate the boundedness of the bilinear
Bochner-Riesz means $S^{\alpha }$ in the non-Banach triangle case. We
improve the corresponding results in \cite{Bern} in two aspects: Our
partition of the non-Banach triangle is simpler and we obtain lower
smoothness indices $\alpha (p_{1},p_{2})$ for various cases apart from $1
\leq p_1=p_2 <2$.
\end{abstract}

\maketitle

\section{Introduction}

The bilinear Bochner-Riesz means problem originates from the study of the
summability of the product of two $n$-dimensional Fourier series. This leads
to the study of the $L^{p_{1}}\times L^{p_{2}}\rightarrow L^{p}$ boundedness
of the bilinear Fourier multiplier operator 
\begin{equation*}
S^{\alpha }(f,g)(x)=\int \int_{\left\vert \xi \right\vert ^{2}+\left\vert
\eta \right\vert ^{2}\leq 1}(1-\left\vert \xi \right\vert ^{2}-\left\vert
\eta \right\vert ^{2})^{\alpha }\widehat{f}(\xi )\widehat{g}(\eta )e^{2\pi
ix\cdot (\xi +\eta )}d\xi d\eta ,
\end{equation*}
where $x\in \mathbb{R}^{n}$, $f$, $g$ are functions on $\mathbb{R}^{n}$ and $%
\widehat{f}$, $\widehat{g}$ are their Fourier transforms. Of course, we hope
to obtain a smoothness index $\alpha (p_{1},p_{2})$ as low as possible such
that $S^{\alpha}$ is bounded from $L^{p_{1}} \times L^{p_{2}}$ into $L^{p}$
for $1 \leq p_1, p_2 \leq \infty$ and $1/p=1/p_{1}+1/p_{2}$ when $\alpha >
\alpha (p_{1},p_{2})$. Bernicot et al. \cite{Bern} gave a comprehensive
study on this problem. Their results are pretty well when $n=1$ and in the
Banach triangle case when $n \geq 2$. In this article, we investigate the $%
L^{p_{1}} \times L^{p_{2}} \rightarrow L^{p}$ boundedness of $S^{\alpha }$
in the non-Banach triangle case, i.e. $1 \leq p_1, p_2 \leq \infty$ and $p<1$%
, when $n \geq 2$. We improve the corresponding results in \cite{Bern} in
two aspects. Firstly, our partition of the non-Banach triangle is simpler.
Secondly, we obtain lower smoothness indices $\alpha (p_{1},p_{2})$ for
various cases apart from $1 \leq p_1=p_2 <2$.

Our results are summarized in the following theorem. The pictures also
display the comparison of two partitions.

\begin{theorem}
\label{mainTh} Assume that $n \geq 2$. Let $1\leq p_{1}, p_{2} \leq \infty $
and $1/p=1/p_{1}+1/p_{2}$.

\textrm{(1)(region I)} For $1\leq p_{1} \leq 2\leq p_{2} \leq \infty $ and $%
p< 1$, if $\alpha >n \big( \frac{1}{p_{1}}-\frac{1}{2}\big)$, then $%
S^{\alpha}$ is bounded from $L^{p_{1}} \times L^{p_{2}}$ to $L^{p}$; For $%
1\leq p_{2} \leq 2\leq p_{1} \leq \infty$ and $p< 1$, if $\alpha >n \big(
\frac{1}{p_{2}}-\frac{1}{2}\big) $, then $S^{\alpha }$ is bounded from $%
L^{p_{1}} \times L^{p_{2}}$ to $L^{p}$.

\textrm{(2)(region II)} For $1\leq p_{1} \leq p_{2} \leq 2$, if $\alpha > n%
\big( \frac{1}{p}-1 \big)- \big( \frac{1}{p_{2}}- \frac{1}{2} \big)$, then $%
S^{\alpha}$ is bounded from $L^{p_{1}} \times L^{p_{2}}$ to $L^{p}$; For $%
1\leq p_{2} \leq p_{1} \leq 2$, if $\alpha >n\big( \frac{1}{p}-1 \big)- %
\big( \frac{1}{p_{1}}- \frac{1}{2} \big)$, then $S^{\alpha}$ is bounded from 
$L^{p_{1}} \times L^{p_{2}}$ to $L^{p}$.
\end{theorem}

\includegraphics[scale=0.85]{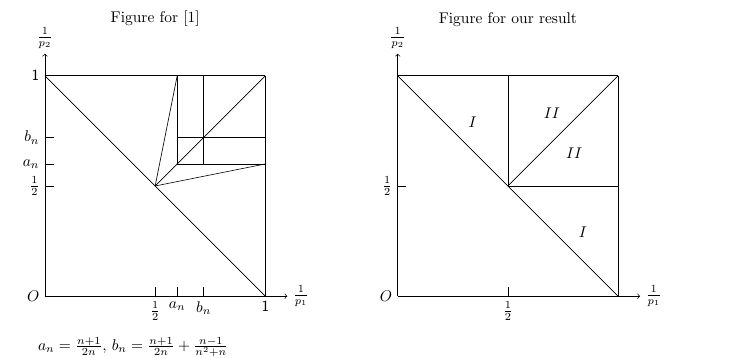}

\textbf{Remark}: Fefferman \cite{Feff} pointed out that the restriction
estimates apply to the study of the boundedness of Bochner-Riesz means.
Stein's earlier result in \cite{Stein} was improved by using the restriction
theorem (see \cite{Stein2}). Our results are obtained without the help of
non-trivial restriction estimates. In contrast, the results in \cite{Bern}
depend heavily on the restriction-extension estimates. It seems still a
problem that how apply the restriction estimates to obtain a better result
about the boundedness of the bilinear Bochner-Riesz means.

\vskip 0.5 cm

\section{Preliminaries}

We state some facts about the Bessel function and the kernel of the
Bochner-Riesz operator. These facts can be found in many references, for
example, in \cite{Stein2}.

We define the restriction operators 
\begin{equation*}
R_{\lambda }f(x)=\int_{\mathbb{S}^{n-1}} \widehat{f}(\lambda \omega )e^{2\pi
ix\cdot \lambda \omega }\, d\sigma (\omega ), \quad \lambda >0,
\end{equation*}
where $d\sigma $ denote the surface measure on $\mathbb{S}^{n-1}$. Then 
\begin{equation*}
R_{1}f(x)= \int_{\mathbb{S}^{n-1}}\widehat{f}(\omega )e^{2\pi ix\cdot \omega
}d\sigma (\omega )= f\ast \widehat{d\sigma },
\end{equation*}
It is known that 
\begin{equation*}
\widehat{d\sigma }(x)=\int_{\mathbb{S}^{n-1}}e^{2\pi ix\cdot \omega }d\sigma
(\omega )=\frac{2\pi }{\left\vert x\right\vert ^{\frac{n-2}{2}}}J_{\frac{n-2 
}{2}}(2\pi \left\vert x\right\vert ),
\end{equation*}
where $J_{k}$ is the Bessel function given by 
\begin{equation*}
J_{k}(r)= \frac{(\frac{r}{2})^k}{\Gamma (k+ \frac{1}{2}) \pi^{\frac{1}{2}}}
\int_{-1}^1 e^{irt} (1-t^2)^{k-\frac{1}{2}}\, dt, \quad k> -\frac{1}{2}.
\end{equation*}
By a dilation argument, we have 
\begin{equation*}
R_{\lambda }f=f\ast \varphi _{\lambda },\quad \lambda >0,
\end{equation*}
where 
\begin{equation*}
\varphi _{\lambda }(x)=\widehat{d\sigma }\left( \lambda x\right) =\frac{2\pi 
}{\left( \lambda \left\vert x\right\vert \right) ^{\frac{n-2}{2}}} J_{\frac{
n-2}{2}}(2\pi \lambda \left\vert x\right\vert ).
\end{equation*}

The bilinear Bochner-Riesz means can be written as 
\begin{equation*}
S_{R}^{\alpha }(f,g)(x)=\int_{0}^{\infty }\int_{0}^{\infty }\left( 1-\frac{
\lambda _{1}^{2}+\lambda _{2}^{2}}{R^{2}}\right) _{+}^{\alpha }R_{\lambda
_{1}}f(x)R_{\lambda _{2}}g(x)\lambda _{1}^{n-1} \lambda _{2}^{n-1}\,
d\lambda _{1}d\lambda _{2}.
\end{equation*}
It is easy to see that 
\begin{equation*}
S_{R}^{\alpha }(f,g)(x)=\int_{\mathbb{R}^{n}}\int_{\mathbb{R}^{n}}
f(x-x_{1}) g(x-x_{2}) S_{R}^{\alpha }(x_{1},x_{2})\, dx_{1}dx_{2},
\end{equation*}
where the kernel is given by 
\begin{equation*}
S_{R}^{\alpha }(x_{1},x_{2})=\int_{0}^{\infty }\int_{0}^{\infty }\left( 1- 
\frac{\lambda _{1}^{2}+\lambda _{2}^{2}}{R^{2}}\right) _{+}^{\alpha }\varphi
_{\lambda _{1}}(x_{1})\varphi _{\lambda _{2}}(x_{2})\lambda
_{1}^{n-1}\lambda _{2}^{n-1}\, d\lambda _{1}d\lambda _{2}.
\end{equation*}
Because 
\begin{equation*}
S_{R}^{\alpha }(x_{1},x_{2})= R^{2n} S_{1}^{\alpha }(Rx_{1},Rx_{2}),
\end{equation*}
by a dilation argument, we know that the $L^{p_{1}} \times
L^{p_{2}}\rightarrow L^{p}$ boundedness of the operator $S_{R}^{\alpha }$ is
deduced from the $L^{p_{1}}\times L^{p_{2}}\rightarrow L^{p}$ boundedness of
the operator $S^{\alpha}:=S_{1}^{\alpha }$ as $1/p=1/p_{1}+1/p_{2}$.

We note that the kernel of the bilinear Bochner-Riesz operator $S^{\alpha}$
on $\mathbb{R}^{n}$ coincides with the kernel of the Bochner-Riesz operator
on $\mathbb{R}^{2n}$. This means 
\begin{equation*}
S^{\alpha }(x_{1},x_{2})= \frac{\Gamma (1+ \alpha)}{\pi^{\alpha} \big |%
(x_{1},x_{2})\big |^{n+ \alpha}} J_{n+ \alpha} \big( 2\pi \big |(x_{1},x_{2})%
\big | \big).
\end{equation*}
The Bessel function satisfies the asymptotic estimate 
\begin{equation*}
J_{k}(r)= O \big( r^{-\frac{1}{2}}\big), \quad r \rightarrow \infty.
\end{equation*}
As a consequence, we have a basic result: $S^{\alpha}$ is bounded from $%
L^{p_{1}}\times L^{p_{2}}$ into $L^{p}$ when $\alpha > n- \frac{1}{2}$,
which is optimal in case of $(p_1,p_2, p)= (1,1, \frac{1}{2})$.

\vskip 0.5 cm

\section{Proof of results.}

\begin{lemma}
\label{restrictionLe} Suppose $m\in L^{\infty }(\mathbb{R})$. Define
operator $T_{m}f=\int_{a}^{b}m(\lambda )R_{\lambda }f\,\lambda
^{n-1}d\lambda $ for $0\leq a<b$. Then, for $1\leq p\leq 2$, we have 
\begin{equation*}
\left\Vert T_{m}f\right\Vert _{2}\leq C\left( (b-a)b^{n-1}\right)^{\frac{1}{2%
}} \left\Vert m\right\Vert _{\infty }\left\Vert f\right\Vert _{p}
\end{equation*}
\end{lemma}

\begin{proof}
\quad The operator $T_{m}$ can be written as 
\begin{equation*}
T_{m}f=f\ast G_{m},
\end{equation*}
where the kernel 
\begin{equation*}
G_{m}(x)=\int_{\mathbb{R}^{n}}\chi _{\lbrack a,b]} (\left\vert \xi
\right\vert )m(\left\vert \xi \right\vert )e^{2\pi i\left\langle x,\xi
\right\rangle }\,d\xi.
\end{equation*}
Applying Plancherel theorem, we have 
\begin{equation*}
\left\Vert G_{m}\right\Vert _{2}^{2}= \int_{a\leq \left\vert \xi \right\vert
\leq b} \left\vert m(\left\vert \xi \right\vert )\right\vert ^{2}\, d\xi
\leq C \left\Vert m\right\Vert _{\infty }^{2}\left( (b-a)b^{n-1}\right).
\end{equation*}
It follows that 
\begin{equation*}
\left\Vert T_{m}f\right\Vert _{2}\leq \left\Vert G_{m}\right\Vert _{2}
\left\Vert f\right\Vert _{1} \leq C \left( (b-a)b^{n-1}\right) ^{\frac{1}{2}
}\left\Vert m\right\Vert _{\infty} \left\Vert f\right\Vert _{1}.
\end{equation*}%
Interpolation with the trivial estimate 
\begin{equation*}
\left\Vert T_{m}f\right\Vert _{2}\leq \left\Vert m\right\Vert _{\infty
}\left\Vert f\right\Vert _{2},
\end{equation*}
we get 
\begin{equation*}
\left\Vert T_{m}f\right\Vert _{2}\leq C\left( (b-a)b^{n-1}\right) ^{(\frac{1%
}{p}-\frac{1}{2})}\left\Vert m\right\Vert _{\infty }\left\Vert f\right\Vert
_{p}.
\end{equation*}
Lemma \ref{restrictionLe} is proved.
\end{proof}

\begin{theorem}
\label{th2} Let $n\geq 2$. Suppose $1\leq p_{1},p_{2}\leq 2$ and $%
1/p=1/p_{1}+1/p_{2}$. If $\alpha >n\big(\frac{1}{p}-1\big)$, then $S^{\alpha
}$ is bounded from $L^{p_{1}}(\mathbb{R}^{n})\times L^{p_{2}}(\mathbb{R}%
^{n}) $ into $L^{p}(\mathbb{R}^{n})$.
\end{theorem}

\begin{proof}
\quad First, we choose a nonnegative function $\varphi $ $\in C_{0}^{\infty
}(\frac{1}{2} ,2)$ satisfying $\sum_{-\infty }^{\infty }\varphi (2^{j}s)=1$, 
$s>0$. For each $j\geq 0$, we set function 
\begin{equation*}
\varphi _{j}^{\alpha }\left( s,t\right) =(1-s^{2}-t^{2})_{+}^{\alpha
}\varphi \left( 2^{j}\left( 1-s^{2}-t^{2}\right) \right),
\end{equation*}
and define bilinear operator 
\begin{equation*}
T_{j}^{\alpha }(f,g)=\int_{0}^{\infty }\int_{0}^{\infty }\varphi
_{j}^{\alpha }\left( \lambda _{1},\lambda _{2}\right) R_{\lambda
_{1}}fR_{\lambda _{2}}g\,\lambda _{1}^{n-1}\lambda _{2}^{n-1}\, d\lambda
_{1}d\lambda _{2}.
\end{equation*}
It is obvious that 
\begin{equation*}
S^{\alpha }=\sum_{j=0}^{\infty }T_{j}^{\alpha },
\end{equation*}
and our result would follow if we can show that when $\alpha >n\big(\frac{1}{
p}-1\big)$, there exists some $\varepsilon >0$ such that for each $j\geq 0$, 
\begin{equation}
\left\Vert T_{j}^{\alpha }\right\Vert _{L^{p_{1}}\times L^{p_{2}}\rightarrow
L^{p}}\leq 2^{-\varepsilon j}.  \label{Tj-}
\end{equation}
Fixing $j\geq 0$. In order to prove (\ref{Tj-}), we define $%
B_{j}=\{x:\left\vert x\right\vert \leq 2^{j(1+\gamma )}\}\subseteq \mathbb{%
\mathbb{R} }^{n}$ and split the kernel $K_{j}^{\alpha }$ of $T_{j}^{\alpha }$
into four parts: 
\begin{equation*}
K_{j}^{\alpha }=K_{j}^{1}+K_{j}^{2}+K_{j}^{3}+K_{j}^{4},
\end{equation*}
where 
\begin{eqnarray*}
K_{j}^{1}(x_{1},x_{2}) &=&K_{j}^{\alpha }(x_{1},x_{2})\chi
_{B_{j}}(x_{1})\chi _{B_{j}}(x_{2}), \\
K_{j}^{2}(x_{1},x_{2}) &=&K_{j}^{\alpha }(x_{1},x_{2})\chi
_{B_{j}}(x_{1})\chi _{B_{j}^{c}}(x_{2}), \\
K_{j}^{3}(x_{1},x_{2}) &=&K_{j}^{\alpha }(x_{1},x_{2})\chi
_{B_{j}^{c}}(x_{1})\chi _{B_{j}}(x_{2}), \\
K_{j}^{4}(x_{1},x_{2}) &=&K_{j}^{\alpha }(x_{1},x_{2})\chi
_{B_{j}^{c}}(x_{1})\chi _{B_{j}^{c}}(x_{2}).
\end{eqnarray*}
Here $\chi _{A}$ stands for the characteristic function of set $A$ and $%
\gamma >0$ is to be fixed. Let $T_{j}^{l}$ be the bilinear operator with
kernel $K_{j}^{l}$, $l=1,2,3,4$. Then, (\ref{Tj-}) would be the consequence
of the estimates 
\begin{equation*}
\left\Vert T_{j}^{l}\right\Vert _{L^{p_{1}}\times L^{p_{2}}\rightarrow
L^{p}}\leq C2^{-\varepsilon j},\quad l=1,2,3,4.
\end{equation*}

We first consider $T_{j}^{4}$. Note that the kernel of $T_{j}^{\alpha }$ is
written as 
\begin{equation*}
K_{j}^{\alpha }(x_{1},x_{2})=\int_{0}^{\infty }\int_{0}^{\infty }\varphi
_{j}^{\alpha }(\lambda _{1},\lambda _{2})\varphi _{\lambda
_{1}}(x_{1})\varphi _{\lambda _{2}}(x_{2})\,\lambda _{1}^{n-1}\lambda
_{2}^{n-1}\,d\lambda _{1}d\lambda _{2}
\end{equation*}%
satisfying 
\begin{equation*}
\left\vert K_{j}^{\alpha }(x_{1},x_{2})\right\vert \leq C2^{-j\alpha
}2^{-j}(1+2^{-j}\left\vert x_{1}\right\vert )^{-M}(1+2^{-j}\left\vert
x_{2}\right\vert )^{-M}
\end{equation*}%
for each $M>0$. So, by H\"{o}lder's inequality and Young's inequality, we
can get 
\begin{eqnarray*}
\left\Vert T_{j}^{4}(f,g)\right\Vert _{p} &\leq &C\left\Vert f\right\Vert
_{p_{1}}\left\Vert g\right\Vert _{p_{2}}\int_{\left\vert x_{1}\right\vert
\geq 2^{j(1+\gamma )}}(1+2^{-j}\left\vert x_{1}\right\vert )^{-M}dx_{1} \\
&&\times \int_{\left\vert x_{2}\right\vert \geq 2^{j(1+\gamma
)}}(1+2^{-j}\left\vert x_{2}\right\vert )^{-M}dx_{2} \\
&\leq &C\left( 2^{jM}2^{-j(1+\gamma )(M-n)}\right) ^{2}\left\Vert
f\right\Vert _{p_{1}}\left\Vert g\right\Vert _{p_{2}}.
\end{eqnarray*}%
Choosing $M$ large enough such that 
\begin{equation*}
M\gamma >(1+\gamma )n,
\end{equation*}%
it follows that 
\begin{equation}
\left\Vert T_{j}^{4}\right\Vert _{L^{p_{1}}\times L^{p_{2}}\rightarrow
L^{p}}\leq C2^{-\varepsilon j}  \label{Tj4}
\end{equation}%
for some $\varepsilon >0$. Similarly, for $T_{j}^{3}$, we have 
\begin{eqnarray*}
\left\Vert T_{j}^{3}(f,g)\right\Vert _{p} &\leq &C\left\Vert f\right\Vert
_{p_{1}}\left\Vert g\right\Vert _{p_{2}}\int_{\left\vert x_{1}\right\vert
\leq 2^{j(1+\gamma )}}(1+2^{-j}\left\vert x_{1}\right\vert )^{-M}dx_{1} \\
&&\times \int_{\left\vert x_{2}\right\vert \geq 2^{j(1+\gamma
)}}(1+2^{-j}\left\vert x_{2}\right\vert )^{-M}dx_{2} \\
&\leq &C2^{j(1+\gamma )n}2^{jM}2^{-j(1+\gamma )(M-n)}\left\Vert f\right\Vert
_{p_{1}}\left\Vert g\right\Vert _{p_{2}}.
\end{eqnarray*}%
Choosing large $M$ satisfying 
\begin{equation*}
M\gamma >2(1+\gamma )n,
\end{equation*}%
it follows that 
\begin{equation}
\left\Vert T_{j}^{3}\right\Vert _{L^{p_{1}}\times L^{p_{2}}\rightarrow
L^{p}}\leq C2^{-\varepsilon j}.  \label{Tj3}
\end{equation}%
for some $\varepsilon >0$. Obviously, (\ref{Tj3}) also holds for $T_{j}^{2}$.

Now, it remains to estimate $T_{j}^{1}$. For any fixed $y\in \mathbb{R} ^{n}$%
, we set $B_{j}(y,R)=\{x:\left\vert x-y\right\vert \leq R2^{j(1+\gamma )}\}$
with $R>0$, and slipt the functions $f$ and $g$ into three parts
respectively: $f=f_{1}+f_{2}+f_{3}$, $g=g_{1}+g_{2}+g_{3}$, where 
\begin{eqnarray*}
f_{1} &=&f\chi _{B_{j}(y,\frac{3}{4})},\qquad \quad\quad\ g_{1}=g\chi
_{B_{j}(y,\frac{3}{4})}, \\
f_{2} &=&f\chi _{B_{j}(y,\frac{5}{4})\backslash B_{j}(y,\frac{3}{4})},\quad
\ g_{2}=g\chi _{B_{j}(y,\frac{5}{4})\backslash B_{j}(y,\frac{3}{4})}, \\
f_{3} &=&f\chi _{\mathbb{R}^{n}\backslash B_{j}(y,\frac{5}{4})},\qquad \ \ \
g_{3}=g\chi _{\mathbb{R}^{n}\backslash B_{j}(y,\frac{5}{4})}.
\end{eqnarray*}
Assume that $\left\vert x-y\right\vert \leq \frac{1}{4}2^{j(1+\gamma )}$.
Since that $f_{3}$ is supported on $\mathbb{R}^{n}\backslash B_{j}(y,\frac{5%
}{4})$, then $f_{3}\neq 0$ leads to 
\begin{equation*}
\left\vert x-x_{1}-y\right\vert \geq \frac{5}{4}2^{j(1+\gamma )}.
\end{equation*}
It follows that 
\begin{equation*}
\left\vert x_{1}\right\vert \geq 2^{j(1+\gamma )}.
\end{equation*}
Note that the kernel $K_{j}^{1}$ is supported on $B_{j}\times B_{j}$. Hence, 
$T_{j}^{1}(f_{3},g)=0$. In the same way, we have $T_{j}^{1}(f,g_{3})=0$.
Since that $f_{2}$ and $g_{2}$ are supported on $B_{j}(y,\frac{5}{4}
)\backslash B_{j}(y,\frac{3}{4})$, then $f_{2},g_{2}\neq 0$ yields that 
\begin{equation*}
\left\vert x_{1}\right\vert \geq \frac{1}{2}2^{j(1+\gamma )}\text{ and }
\left\vert x_{2}\right\vert \geq \frac{1}{2}2^{j(1+\gamma )}.
\end{equation*}
Repeating the proof of (\ref{Tj4}), we get that 
\begin{eqnarray*}
\left\Vert T_{j}^{1}(f_{2},g_{2})\right\Vert _{L^{p}(B_{j}(y,\frac{1}{4}))}
&\leq &C2^{-\varepsilon j}\left\Vert f_{2}\right\Vert _{p_{1}}\left\Vert
g_{2}\right\Vert _{p_{2}} \\
&\leq &C2^{-\varepsilon j}\left\Vert f\right\Vert _{L^{p_{1}}(B_{j}(y,\frac{%
5 }{4}))}\left\Vert g\right\Vert _{L^{p_{2}}(B_{j}(y,\frac{5}{4}))}.
\end{eqnarray*}
Both side of this inequality are the fuction of $y$. So, taking the $L^{p}$
norm with respect to $y$ and using H\"{o}lder's inequality, we have that 
\begin{eqnarray}
&&\left( \int_{\mathbb{\mathbb{R}}^{n}}\int_{B_{j}(y,\frac{1}{4})}\left\vert
T_{j}^{1}(f_{2},g_{2})(x)\right\vert ^{p}dxdy\right) ^{\frac{1}{p}}\notag\\
&\leq &C2^{-\varepsilon j}\left( \int_{\mathbb{R}^{n}}\int_{B_{j}(y,\frac{5}{
4})}\left\vert f(x)\right\vert ^{p_{1}}dxdy\right) ^{\frac{1}{p_{1}}}\left(
\int_{\mathbb{R}^{n}}\int_{B_{j}(y,\frac{5}{4})}\left\vert g(x)\right\vert
^{p_{2}}dxdy\right) ^{\frac{1}{p_{2}}}.  \notag
\end{eqnarray}
Changing variable and exchanging the order of integration, the left side
equals to 
\begin{eqnarray*}
&&\left( \int_{\mathbb{\mathbb{R}}^{n}}\int_{B_{j}(y,\frac{1}{4})}\left\vert
T_{j}^{1}(f_{2},g_{2})(x)\right\vert ^{p}dxdy\right) ^{\frac{1}{p}} \\
&=&\left( \int_{\mathbb{R}^{n}}\int_{\left\vert x\right\vert \leq \frac{1}{4}%
2^{j(1+\gamma )}}\left\vert T_{j}^{1}(f_{2},g_{2})(x+y)\right\vert
^{p}dxdy\right) ^{\frac{1}{p}} \\
&=&\left( \int_{\left\vert x\right\vert \leq \frac{1}{4}2^{j(1+\gamma
)}}\int_{\mathbb{R}^{n}}\left\vert T_{j}^{1}(f_{2},g_{2})(x+y)\right\vert
^{p}dydx\right) ^{\frac{1}{p}} \\
&=&\left( \frac{1}{4}2^{j(1+\gamma )}\right) ^{\frac{n}{p}}\left\Vert
T_{j}^{1}(f_{2},g_{2})\right\Vert _{p}.
\end{eqnarray*}
Similarly, the right side equals to 
\begin{equation*}
C2^{-\varepsilon j}\left( \frac{5}{4}2^{j(1+\gamma )}\right) ^{\frac{n}{%
p_{1} }}\left\Vert f\right\Vert _{p_{1}}\left( \frac{5}{4}2^{j(1+\gamma
)}\right) ^{\frac{n}{p_{2}}}\left\Vert g\right\Vert
_{p_{2}}=C2^{-\varepsilon j}\left( \frac{5}{4}2^{j(1+\gamma )}\right) ^{%
\frac{n}{p}}\left\Vert f\right\Vert _{p_{1}}\left\Vert g\right\Vert _{p_{2}}.
\end{equation*}
This yields that 
\begin{equation}
\left\Vert T_{j}^{1}(f_{2},g_{2})\right\Vert _{p}\leq C2^{-\varepsilon
j}\left\Vert f\right\Vert _{p_{1}}\left\Vert g\right\Vert _{p_{2}}.
\label{f3}
\end{equation}
Since that $f_{1}$ is supported on $B_{j}(y,\frac{3}{4})$, then $%
f_{1},g_{2}\neq 0$ implies that 
\begin{equation*}
\left\vert x_{1}\right\vert \leq 2^{j(1+\gamma )}\text{ and }\left\vert
x_{2}\right\vert \geq \frac{1}{2}2^{j(1+\gamma )}.
\end{equation*}
Repeating the proof of (\ref{Tj3}), we have 
\begin{eqnarray*}
\left\Vert T_{j}^{1}(f_{1},g_{2})\right\Vert _{L^{p}(B_{j}(y,\frac{1}{4} ))}
&\leq &C2^{-\varepsilon j}\left\Vert f_{1}\right\Vert _{p_{1}}\left\Vert
g_{2}\right\Vert _{p_{2}} \\
&\leq &C2^{-\varepsilon j}\left\Vert f\right\Vert _{L^{p_{1}}(B_{j}(y, \frac{%
3}{4}))}\left\Vert g\right\Vert _{L^{p_{2}}(B_{j}(y,\frac{5}{4}))}.
\end{eqnarray*}
Taking the $L^{p}$ norm with respect to $y$ as above, we get that 
\begin{equation}
\left\Vert T_{j}^{1}(f_{1},g_{2})\right\Vert _{p}\leq C2^{-\varepsilon
j}\left\Vert f\right\Vert _{p_{1}}\left\Vert g\right\Vert _{p_{2}}.
\label{f2}
\end{equation}
Obviously, (\ref{f2}) also holds for $T_{j}^{1}(f_{2},g_{1})$.

Finally, we consider $T_{j}^{1}(f_{1},g_{1})$. Because $f_{1},g_{1}\neq 0$
implies that 
\begin{equation*}
\left\vert x_{1}\right\vert \leq 2^{j(1+\gamma )}\text{ and }\left\vert
x_{2}\right\vert \leq 2^{j(1+\gamma )},
\end{equation*}%
we have 
\begin{equation}
T_{j}^{1}(f_{1},g_{1})(x)=T_{j}^{\alpha }(f_{1},g_{1})(x),\quad x\in B_{j}%
\big(y,\frac{1}{4}\big).  \label{f1}
\end{equation}%
Note that $T_{j}^{\alpha }$ can be written as 
\begin{eqnarray*}
T_{j}^{\alpha }(f,g) &=&\int_{0}^{\infty }\int_{0}^{\infty }\varphi
_{j}^{\alpha }\left( \lambda _{1},\lambda _{2}\right) R_{\lambda
_{1}}fR_{\lambda _{2}}g\,\lambda _{1}^{n-1}\lambda _{2}^{n-1}d\lambda
_{1}d\lambda _{2} \\
&=&\frac{1}{2}\int_{[-1,1]^{2}}\varphi _{j}^{\alpha }\left( \left\vert
\lambda _{1}\right\vert ,\left\vert \lambda _{2}\right\vert \right)
R_{\left\vert \lambda _{1}\right\vert }fR_{\left\vert \lambda
_{2}\right\vert }g\,\left\vert \lambda _{1}\right\vert ^{n-1}\left\vert
\lambda _{2}\right\vert ^{n-1}d\lambda _{1}d\lambda _{2}\text{.}
\end{eqnarray*}%
Because for any fixed $s\in \lbrack -1,1]$, the function 
\begin{equation*}
t\rightarrow \varphi _{j}^{\alpha }\left( \left\vert s\right\vert
,\left\vert t\right\vert \right)
\end{equation*}%
is supported in $[-1,1]$ and vanishes at endpoints $\pm 1$, so we can expand
this function in Fourier series by considering a periodic extension on $%
\mathbb{R}$ of period $2$. Then, we have 
\begin{equation*}
\varphi _{j}^{\alpha }(\left\vert s\right\vert ,\left\vert t\right\vert
)=\sum_{k\in \mathbb{Z}}\gamma _{j,k}^{\alpha }(s)e^{i\pi kt}
\end{equation*}%
with Fourier coefficients 
\begin{equation*}
\gamma _{j,k}^{\alpha }(s)=\frac{1}{2}\int_{-1}^{1}\varphi _{j}^{\alpha
}(\left\vert s\right\vert ,\left\vert t\right\vert )e^{-i\pi kt}\,dt\text{.}
\end{equation*}%
It is easy to see that for any $0<\delta <\alpha $, 
\begin{equation*}
\sup_{s\in \lbrack -1,1]}\left\vert \gamma _{j,k}^{\alpha }(s)\right\vert
\left( 1+\left\vert k\right\vert \right) ^{1+\delta }\leq C2^{-j(\alpha
-\delta )}.
\end{equation*}%
$T_{j}^{\alpha }$ can be expressed by 
\begin{eqnarray*}
T_{j}^{\alpha }(f,g) &=&C\int_{[-1,1]^{2}}\varphi _{j}^{\alpha }\left(
\left\vert \lambda _{1}\right\vert ,\left\vert \lambda _{2}\right\vert
\right) R_{\left\vert \lambda _{1}\right\vert }fR_{\left\vert \lambda
_{2}\right\vert }g\,\left\vert \lambda _{1}\right\vert ^{n-1}\left\vert
\lambda _{2}\right\vert ^{n-1}d\lambda _{1}d\lambda _{2} \\
&=&C\sum_{k\in \mathbb{Z}}\int_{[-1,1]^{2}}\gamma _{j,k}^{\alpha }(\lambda
_{1})e^{i\pi k\lambda _{2}}R_{\left\vert \lambda _{1}\right\vert
}fR_{\left\vert \lambda _{2}\right\vert }g\,\left\vert \lambda
_{1}\right\vert ^{n-1}\left\vert \lambda _{2}\right\vert ^{n-1}d\lambda
_{1}d\lambda _{2} \\
&=&C\sum_{k\in \mathbb{Z}}\int_{-1}^{1}\gamma _{j,k}^{\alpha }(\lambda
_{1})R_{\left\vert \lambda _{1}\right\vert }f\,\left\vert \lambda
_{1}\right\vert ^{n-1}\,d\lambda _{1}\int_{-1}^{1}e^{i\pi k\lambda
_{2}}R_{\left\vert \lambda _{2}\right\vert }g\,\left\vert \lambda
_{2}\right\vert ^{n-1}d\lambda _{2}.
\end{eqnarray*}%
Applying Cauchy-Schwartz's inequality and Lemma \ref{restrictionLe}, we get
that 
\begin{eqnarray}
&&\left\Vert T_{j}^{\alpha }(f,g)\right\Vert _{1}  \label{f5} \\
&\leq &C\sum_{k\in \mathbb{Z}}\left\Vert \int_{-1}^{1}\gamma _{j,k}^{\alpha
}(\lambda _{1})R_{\left\vert \lambda _{1}\right\vert }f\,\left\vert \lambda
_{1}\right\vert ^{n-1}\,d\lambda _{1}\right\Vert _{2}\left\Vert
\int_{-1}^{1}e^{i\pi k\lambda _{2}}R_{\left\vert \lambda _{2}\right\vert
}g\,\left\vert \lambda _{2}\right\vert ^{n-1}d\lambda _{2}\right\Vert _{2} 
\notag \\
&\leq &C\sum_{k\in \mathbb{Z}}(1+\left\vert k\right\vert )^{-1-\delta
}\left( \sup_{s\in \lbrack -1,1]}\left\vert \gamma _{j,k}^{\alpha
}(s)\right\vert \left( 1+\left\vert k\right\vert \right) ^{1+\delta }\right)
\left\Vert f\right\Vert _{p_{1}}\left\Vert g\right\Vert _{p_{2}}.  \notag \\
&\leq &C2^{-j(\alpha -\delta )}\left\Vert f\right\Vert _{p_{1}}\left\Vert
g\right\Vert _{p_{2}}.  \notag
\end{eqnarray}%
Using H\"{o}lder's inequality and (\ref{f1}), it follows that 
\begin{eqnarray*}
\left\Vert T_{j}^{1}(f_{1},g_{1})\right\Vert _{L^{p}(B_{j}(y,\frac{1}{4}))}
&\leq &2^{jn(1+\mathbb{\gamma })(\frac{1}{p}-1)}\left\Vert
T_{j}^{1}(f_{1},g_{1})\right\Vert _{L^{1}(B_{j}(y,\frac{1}{4}))} \\
&=&2^{jn(1+\mathbb{\gamma })(\frac{1}{p}-1)}\left\Vert T_{j}^{\alpha
}(f_{1},g_{1})\right\Vert _{L^{1}(B_{j}(y,\frac{1}{4}))} \\
&\leq &C2^{-j(\alpha -\delta )}2^{jn(1+\mathbb{\gamma })(\frac{1}{p}%
-1)}\left\Vert f_{1}\right\Vert _{p_{1}}\left\Vert g_{1}\right\Vert _{p_{2}}
\\
&\leq &C2^{-j(\alpha -\delta )}2^{jn(1+\mathbb{\gamma })(\frac{1}{p}%
-1)}\left\Vert f\right\Vert _{L^{p_{1}}(B_{j}(y,\frac{3}{4}))}\left\Vert
g\right\Vert _{L^{p_{2}}(B_{j}(y,\frac{3}{4}))}.
\end{eqnarray*}%
Taking the $L^{p}$ norm with respect to $y$ yields that 
\begin{equation}
\left\Vert T_{j}^{1}(f_{1},g_{1})\right\Vert _{p}\leq C2^{-j(\alpha -\delta
)}2^{jn(1+\mathbb{\gamma })(\frac{1}{p}-1)}\left\Vert f\right\Vert
_{p_{1}}\left\Vert g\right\Vert _{p_{2}}.  \label{f6}
\end{equation}%
Combining (\ref{f3}), (\ref{f2}) and (\ref{f6}), we conclude that 
\begin{equation*}
\left\Vert T_{j}^{1}(f,g)\right\Vert _{p}\leq C2^{-j(\alpha -\delta
)}2^{jn(1+\mathbb{\gamma })(\frac{1}{p}-1)}\left\Vert f\right\Vert
_{p_{1}}\left\Vert g\right\Vert _{p_{2}}.
\end{equation*}%
Therefore, whenever $\alpha >n\big(\frac{1}{p}-1\big)$, we can choose $%
\mathbb{\gamma },\delta >0$ such that 
\begin{equation*}
\alpha >n(1+\gamma )\Big(\frac{1}{p}-1\Big)+\delta ,
\end{equation*}%
which implies that there exists an $\varepsilon >0$ such that 
\begin{equation*}
\left\Vert T_{j}^{1}\right\Vert _{L^{p_{1}}\times L^{p_{2}}\rightarrow
L^{p}}\leq C2^{-\varepsilon j}.
\end{equation*}%
The proof of Theorem \ref{th2} is completed.
\end{proof}

As a consequence of Theorem \ref{th2}, we can give another proof of the
estimate in case of $(p_1,p_2,p)=(1, \infty, 1)$, which was already obtained
in \cite{Bern}.

\begin{corollary}
\label{1infty} If $\alpha >\frac{n}{2}$, then $S^{\alpha}$ is bounded from $%
L^{1} \times L^{\infty}$ to $L^{1}$.
\end{corollary}

\begin{proof}
\quad We keep the notations in the proof of Theorem \ref{th2}. The proof of
Theorem \ref{th2} is valid apart from the estimate of $%
T_{j}^{1}(f_{1},g_{1}) $. According to (\ref{f5}), for any $0< \delta <
\alpha$, 
\begin{equation*}
\left\Vert T_{j}^{\alpha }(f,g)\right\Vert _{1}\leq C2^{-j(\alpha -\delta
)}\left\Vert f\right\Vert _{1}\left\Vert g\right\Vert _{2},
\end{equation*}
we have 
\begin{eqnarray*}
\left\Vert T_{j}^{1}(f_{1},g_{1})\right\Vert _{L^{1}(B_{j}(y,\frac{1}{4} ))}
&=&\left\Vert T_{j}^{\alpha }(f_{1},g_{1})\right\Vert _{L^{1}(B_{j}(y, \frac{%
1 }{4}))} \\
&\leq &C2^{-j(\alpha -\delta)}\left\Vert f\right\Vert _{L^{1}(B_{j}(y, \frac{%
3}{4} ))}\left\Vert g\right\Vert _{L^{2}(B_{j}(y,\frac{3}{4}))} \\
&\leq &C2^{-j(\alpha -\delta)}2^{j(1+\gamma )\frac{n}{2}}\left\Vert
f\right\Vert _{L^{1}(B_{j}(y,\frac{3}{4}))}\left\Vert g\right\Vert
_{L^{\infty }(B_{j}(y,\frac{1}{4}))}.
\end{eqnarray*}
It follows that 
\begin{equation*}
\left\Vert T_{j}^{1}(f_{1},g_{1})\right\Vert _{1}\leq C2^{-j(\alpha -\delta
)}2^{j(1+\gamma )\frac{n}{2}}\left\Vert f\right\Vert _{1}\left\Vert
g\right\Vert _{\infty }.
\end{equation*}
Thus, when $\alpha >\frac{n}{2}$, we can choose $\gamma, \delta >0$ such
that $\alpha >\frac{(1+\gamma)n}{2}+\delta $, which yields that there exists 
$\varepsilon >0$ such that 
\begin{equation*}
\left\Vert T_{j}^{1}(f_{1},g_{1})\right\Vert _{1}\leq C2^{-\varepsilon
j}\left\Vert f\right\Vert _{1}\left\Vert g\right\Vert _{\infty }.
\end{equation*}
The proof is completed.
\end{proof}

Now we have obtained the estimate for $S^{\alpha}$ at some specific triples
of points $(p_{1},p_{2},p)$ like 
\begin{equation*}
(1,1,\frac{1}{2}),(1,2,\frac{2}{3}),(2,1,\frac{2}{3}),(2,2,1),(1,\infty
,1),(\infty ,1,1).
\end{equation*}
In fact, our new result is essentially the estimate in case of $%
(p_{1},p_{2},p)= (1,2,\frac{2}{3})$. The results in Theorem \ref{mainTh} can
be obtained by using of the bilinear interpolation via complex method
adapted to the setting of analytic families or real method in \cite{Graf},
which was described in \cite{Bern}.

\vskip 0.5 cm

\textbf{Acknowledgements} {\quad The first author is supported by National
Natural Science Foundation of China (Grant No. 11371036). The second author
is supported by China Scholarship Council (Grant No. 201606010026). }

\end{document}